\documentclass[11pt]{amsart}
\usepackage{amsfonts,amsmath,amsthm,amssymb,
color,amscd}
\usepackage{fullpage}

\numberwithin{equation}{section}

\newtheorem{thm}{Theorem}[section]
\newtheorem*{thm*}{Theorem}
\newtheorem{prop}[thm]{Proposition}
\newtheorem*{prop*}{Proposition}

\newtheorem{cor}[thm]{Corollary}

\newtheorem{defin}[thm]{Definition}

\newtheorem{lemma}[thm]{Lemma}
\newtheorem*{lemma*}{Lemma}

\newtheorem{remark}[thm]{Remark}
\newtheorem*{remark*}{Remark}
\newtheorem*{remarks*}{Remarks}

\begin{document}

\title[Max. Symm. and Unimod. Solv.]{Maximal Symmetry and unimodular solvmanifolds}
\author[Michael Jablonski]{Michael Jablonski
}
\thanks{%MSC2010: 53C25, 53C30, 53C44, 22E25\\
	This work was supported by a grant from the Simons Foundation (\#360562, michael jablonski) and NSF grant DMS-1612357.}

\begin{abstract}  Recently, it was shown that Einstein solvmanifolds have maximal symmetry in the sense that their isometry groups contain the isometry groups of any other left-invariant metric on the given Lie group.  Such a solvable Lie group is necessarily non-unimodular.  In this work we consider unimodular solvable Lie groups and prove that there is always some metric with maximal symmetry.  Further, if the group at hand admits a Ricci soliton, then it is the isometry group of the  Ricci soliton which is maximal.

\end{abstract}

\maketitle

\section{Introduction}

In the search for  distinguished geometries on a given manifold, one might look towards those with large isometry groups.  Even further, we might ask when a manifold admits a metric of maximal symmetry in the sense that the isometry group of any other metric is a subgroup of the isometry group of our given metric.  At best, of course, such a property could only hold up to conjugation by diffeomorphisms.  We say a manifold is a \textit{maximal symmetry space} if there is a unique largest isometry group up to conjugation.

Clearly constraints must be placed on the manifold for it to be a maximal symmetric space.  For example, $\mathbb H^n$ is diffeomorphic to $\mathbb R^n$, but there is no Lie group which can be the isometry group of a metric and contain both $Isom(\mathbb H^n)$ and $Isom(\mathbb R^n)$ as subgroups.  This follows from the fact that for an n-manifold, the isometry group has dimension bounded above by $n+\frac{1}{2} n(n-1)$ and that the Levi subgroups (i.e. maximal semi-simple subgroups) of $Isom(\mathbb H^n)$ and $Isom(\mathbb R^n)$ are not equal.

An example of a manifold which is a maximal symmetry space is $S^2$, but not all spheres have this property.  So it seems that one should either constrain the topology of a space or restrict to certain classes of metrics if one hopes to learn which types of symmetry groups can arise and how to organize them.   Instead of restricting the topology of $M$, we restrict ourselves to the setting of homogeneous metrics and even more so to just Lie groups with left-invariant metrics.

\begin{defin}\label{def:  max symm}  Let $G$ be a Lie group.  A left-invariant metric $g$ on $G$ is said to be maximally symmetric if given any other left-invariant metric $g'$, there exists a diffeomorphism $\phi \in \mathfrak{Diff}(G)$ such that
	$$Isom(M,g') \subset Isom(M, \phi^* g) =  \phi Isom(M,g) \phi^{-1},$$
where the above inclusion is as a subgroup.  We say $G$ is a maximal symmetry space if it admits a metric of maximal symmetry.
\end{defin}

Although our primary interest is in solvable Lie groups with left-invariant metrics, we briefly discuss the more general setting of Lie groups.  For $G$ compact and simple,  we have that   $Isom(G)_0$, the connected component of of the identity, for any left-invariant metric, can be embedded into the isometry group of the bi-invariant metric \cite{Ochiai-Takahashi:TheGroupOfIsometriesOfALeftInvariantRiemannianMetric}.  This does not quite say that compact simple Lie groups are maximal symmetry spaces, but it is close.

In the setting of non-compact semi-simple groups, one does not have a bi-invariant metric, but there is a natural choice which plays the role of the bi-invariant metric and similar results are known, see  \cite{Gordon:RiemannianIsometryGroupsContainingTransitiveReductiveSubgroups}; note the work of Gordon actually goes beyond the Lie group setting and considers a larger class of homogeneous spaces with transitive reductive Lie group and studies their isometry groups.

If $S$ is a nilpotent or completely solvable unimodular group, then it is a maximal symmetry space.  Although not stated in this language, this is a result of Gordon and Wilson \cite{GordonWilson:IsomGrpsOfRiemSolv}; see Section \ref{sec: preimlin} below for more details.  Furthermore, when such a Lie group admits a Ricci soliton, the soliton metric has the maximal isometry group \cite{Jablo:ConceringExistenceOfEinstein}.

In the non-unimodular setting, it is known that not all solvable groups can be maximal symmetry spaces, however in special circumstances they are - e.g. if a solvable group admits an Einstein metric, then it is a maximal symmetry space and the Einstein metric actually has the largest isometry group; see \cite{GordonJablonski:EinsteinSolvmanifoldsHaveMaximalSymmetry} for more details.

Our main result is for unimodular solvable Lie groups.
\begin{thm}\label{thm: main result}  Let $S$ be a simply-connected, unimodular solvable Lie group.  Then $S$ is a maximal symmetry space.
\end{thm}

\begin{cor}\label{cor: main cor on ricci solitons} Let $S$ be a simply-connected, unimodular solvable Lie group which admits a Ricci soliton metric.  Then said Ricci soliton has maximal symmetry among $S$-invariant metrics.
\end{cor}

The strategy for proving both results is to reduce to the setting of completely solvable groups, where the answer is immediate.   To do this, we start with a solvable group $R$, we modify a given initial metric until we obtain a metric whose isometry group contains a transitive solvable $S$ which is completely solvable.  Our main contribution, then, is to prove a uniqueness result for which $S$ can appear; up to isomorphism only one can and does appear. This uniqueness result is a consequence of the following, which is of independent interest, see Lemma \ref{lemma: modification of completely solvable is normal modification}.  (Here we use the language of \cite{GordonWilson:IsomGrpsOfRiemSolv}.)

\begin{lemma*}
Any modification of a completely solvable group is necessarily a normal modification.
\end{lemma*}

It seems noteworthy to point out that our work actually shows that any solvable Lie group is associated to a unique completely solvable group (Theorem \ref{thm: unique completely solv associated to any solv}) in the same way that type R groups have a well-defined, unique nilshadow, cf.\ \cite{Auslander-Green:G-inducedFlows}.

In the last section we give a concrete description of the completely solvable group associated to any solvable group $S$ in terms of $S$ and the derivations of its Lie algebra.

Finally, we observe that  the  choices made throughout our process allow us to choose our diffeomorphism $\phi$ from Definition \ref{def:  max symm} to be a composition of an automorphism of $S$ together with an automorphism of its associated completely solvable group. In the case that $S$ is completely solvable, the diffeomorphism which conjugates the isometry groups can be chosen to be an automorphism.  It would be interested to know whether or not this is true in general.

\section{Preliminaries}\label{sec: preimlin}

In this section, we recall the basics on isometry groups for (unimodular) solvmanifolds from the foundational work of Gordon and Wilson \cite{GordonWilson:IsomGrpsOfRiemSolv}.  Throughout, our standing assumption is that our solvable groups are simply-connected.  We begin with a general result for Lie algebras.

Recall that every Lie algebra has a unique, maximal  solvable ideal, called the radical.  A (solvable) Lie algebra $\mathfrak g$  is called completely solvable if $ad_X: \mathfrak g \to \mathfrak g$ has only real eigenvalues for all $X\in\mathfrak g$.  We have the following.

\begin{prop}\label{prop: max comp solv ideal} Given any Lie algebra $\mathfrak g$ there exists a unique maximal ideal  $\mathfrak s$ which is  completely solvable.
\end{prop}

\begin{remark}  This completely solvable ideal is contained in the radical, but generally they are not equal.   Notice that the nilradical of $\mathfrak g$ is contained in $\mathfrak s$ and so, as with the radical, $\mathfrak s$ is trivial precisely when $\mathfrak g$ is semi-simple.
\end{remark}

\begin{proof}[Proof of proposition]
As any solvable ideal is a subalgebra of the radical of $\mathfrak g$, it suffices to prove the result in the special case that $\mathfrak g$ is solvable.  The result follows upon showing that the sum of two such ideals is again an ideal of the same type.  As the sum of ideals is again an ideal, we only need to check the condition of complete solvability.

Let $\mathfrak g$ be a solvable Lie algebra and $\mathfrak s_1, \mathfrak s_2$ completely solvable ideals of $\mathfrak g$.  We will show that $\mathfrak s_1 + \mathfrak s_2$ is again completely solvable.  Observe that for any ideal $\mathfrak s$, the eigenvalues of $ad~X|_\mathfrak s$ are real if and only if the eigenvalues of $ad~X|_\mathfrak g$ are real, as we have only introduced extra zero eigenvalues.

Consider $ad~\mathfrak g \subset Der(\mathfrak g)$.  As $ad~\mathfrak g$ is a solvable ideal, it is contained in $rad(Der(\mathfrak g))$, the radical of $Der(\mathfrak g)$.  As $rad(Der(\mathfrak g))$ is an algebraic Lie algebra, it may be decomposed as $rad(Der(\mathfrak g)) = \mathfrak m \ltimes \mathfrak n$, where $\mathfrak m$ is a reductive subalgebra (and hence abelian) and $\mathfrak n$ is the nilradical (see \cite{Mostow:FullyReducibleSubgrpsOfAlgGrps}).  Thus, for $X\in\mathfrak g$, we have
	$$ad~X = M_X + N_X, \mbox{ where } M_X\in \mathfrak m, N_X\in \mathfrak n$$
Applying a strong version of Lie's Theorem (cf. \cite[Section 2]{Jablo:StronglySolvable}), we see that the eigenvalues of $ad~X$ are precisely the eigenvalues of $M_X$.

Now take $X_1\in\mathfrak s_1$ and $X_2\in \mathfrak s_2$.  Decomposing $ad~X_1 = M_1 + N_1$ and $ad~X_2 = M_2+N_2$ as above, we have
	$$ad~(X_1+X_2) = (M_1+M_2) + (N_1+N_2)$$
As $\mathfrak m$ is abelian, we see that the eigenvalues of $ad~(X_1+X_2)$ are the sums of eigenvalues coming from $M_1$ and $M_2$ and thus $\mathfrak s_1+\mathfrak s_2$ is completely solvable.

\end{proof}

\subsection{Isometry groups and modifications}
In the work \cite{GordonWilson:IsomGrpsOfRiemSolv}, Gordon and Wilson set about the job of giving a description of the full isometry group of any solvmanifold.  Given any Lie group $R$ with left-invariant metric, one can build a group of isometries as follows:  let $C$ denote the set of orthogonal automorphisms of $R$, then $R\rtimes C$ is a subgroup of the isometry group.  We call this group the algebraic isometry group and denote it by $AlgIsom(R,g)$.

For $R$ nilpotent, this gives the full isometry group \cite[Cor.~4.4]{GordonWilson:IsomGrpsOfRiemSolv}.  However, in general, the isometry group $Isom(R,g)$ will be much more.  A good example of this is to look at a symmetric space of non-compact type.

So to understand the general setting, Gordon and Wilson detail   a  process of modifying the initial solvable group $R$ to one with a `better' presentation $R'$  called a standard modification of $R$ - this is another solvable group of isometries which acts transitively.  The modification process ends after (at most) two normal modifications with the solvable group $R''$ in so-called standard position.  See Section 3, loc.~cit., for details.

To illustrate why this process is nice, we present the following result in the case of unimodular, solvable Lie groups.

\begin{lemma}\label{lemma: isom group of unimodular solv}  Let $R$ be a unimodular solvable Lie group with left-invariant metric $g$.  Then
  $Isom(R,g) = AlgIsom(R'',g) = C \ltimes R''$ where $R''$ is the solvable group in standard position inside $Isom(R,g)$ and $C$ consists of orthogonal automorphisms of $\mathfrak r''$.
\end{lemma}

This follows from the following  facts proven in Theorem 4.2 and 4.3 of \cite{GordonWilson:IsomGrpsOfRiemSolv}.

\begin{prop} If there is one transitive solvable Lie group of isometries which is unimodular, then all transitive solvable groups of isometries are unimodular. 
\end{prop}

\begin{prop}\label{prop:  compl solv are in std position}    If $R$ is solvable,  unimodular, and in standard position, then the isometry group is the algebraic isometry group.  
\end{prop}

\begin{remark}  Completely solvable groups are always in standard position.
\end{remark}

Regarding normal modifications, we record the following useful facts here.

\begin{lemma}\label{lemma: norman modification goes both ways}   For solvable Lie groups $R$ and $S$ in a common isometry group, $R$ being a normal modification of $S$ implies $S$ is a normal modification of $R$.
\end{lemma}

This follows immediately from the description of normal modifications given in Proposition 2.4 of \cite{GordonWilson:IsomGrpsOfRiemSolv}.  This will be used in the sequel when $S$ is completely solvable.  Such $S$ are in standard position in the isometry group and any modification $R$ is a normal modification (see Lemma \ref{lemma: modification of completely solvable is normal modification}), so we see that there exists an abelian subalgebra $\mathfrak t$ of the stabilizer subalgebra which normalizes both $\mathfrak r$ and $\mathfrak s$ such that $\mathfrak s \subset \mathfrak r \rtimes \mathfrak t$ and $\mathfrak r \subset \mathfrak s \rtimes \mathfrak t$, cf.~Theorem 3.1 of \cite{GordonWilson:IsomGrpsOfRiemSolv}.  As such, we have the following.

\begin{lemma}\label{lemma: s is an ideal in r rtimes N_l(r)}   For $\mathfrak s$ a completely solvable algebra in the isometry algebra and $\mathfrak r$ a modification of $\mathfrak s$, we have $[\mathfrak s,\mathfrak r] \subset \mathfrak s \cap \mathfrak r$.
\end{lemma}

\subsection{Transitive groups of isometries}
The following technical lemmas will be needed later.

\begin{lemma}\label{lemma: changing transitive group of isoms}
Consider a solvable Lie group with left-invariant metric $(R_1,g)$ with isometry group $G=Isom(R_1,g)$.  Suppose $R_2$ is a subgroup of isometries.  Then $R_2$ acts transitively if and only if $\mathfrak g = \mathfrak l + \mathfrak r_2$, where $\mathfrak l$ is the isotropy subalgebra at some point.
\end{lemma}

\begin{proof}
If $R_2$ acts transitively, then we immediately see that $G=LR_2$ for any isotropy group $L$, which implies $\mathfrak g = \mathfrak l + \mathfrak r_2$.  Conversely, if the equality $\mathfrak g = \mathfrak l + \mathfrak r_2$ holds, then the orbit of $R_2$ is open and so is transitive since is a complete submanifold of maximal dimension, cf.~  \cite[Lemma 3.8]{Jablo:StronglySolvable}.
\end{proof}

\begin{remark}
We  apply this lemma repeatedly in the following special case. Let $K$ denote the orthogonal automorphisms of $R_1$; this is a subgroup of the isotropy group which fixes $e\in R_1$.  If $\mathfrak k + \mathfrak r_2 \supset \mathfrak r_1$, then $R_2$ acts transitively.
\end{remark}

\section{Proof of main result in the special case of completely solvable and unimodular}
Our general strategy is to reduce to the case that the group is completely solvable and so we begin here.  Let $S$ be  unimodular and completely solvable.  For the sake of consistency throughout the later sections we write  $R=S$ in this section.

\begin{thm}\label{thm: max symm for compl solv unimod}   Let $S$ be a simply-connected, unimodular, completely solvable Lie group.  Then $S$ is a maximal symmetry space.
\end{thm}

This theorem is an immediate consequence of the following result of Gordon-Wilson, as we see below.

\begin{thm}[Gordon-Wilson] Let $S$ be a simply-connected, unimodular, completely solvable Lie group with left-invariant metric $g$.  Then $Isom(S,g) = S\rtimes C$, where $C = Aut(\mathfrak s) \cap O(g)$.
\end{thm}

\begin{remark*} In the above, we have abused notation as we are viewing $C$ as a subgroup of $Aut(S)$.  This is OK as  $S$ being simply-connected gives that the action of $C$ on $\mathfrak s$ lifts to an action on $S$.
\end{remark*}

Given the compact group $C<Aut(S)$, choose any maximal compact subgroup $K$ of $Aut(S)$ containing $C$ and an inner product $g'$ on $\mathfrak s$ so that $K$ acts orthogonally.  Now we have 
	$$Isom(S,g) < Isom(S,g')$$
To see that $S$ is indeed a homogeneous maximal symmetry space, we only need to compare isometry groups where $C=K$ is a maximal compact subgroup of $Aut(S)$.  

Let $g_1$ and $g_2$ be two left-invariant metrics with isometry groups $S\rtimes K_1$ and $S\rtimes K_2$, respectively, such that $K_1$ and $K_2$ are maximal compact subgroups of automorphisms.  As maximal compact subgroups are all conjugate, there exists $\phi \in Aut(S)$ such that $K_1 = \phi K_2 \phi^{-1}$ and hence 
	$$Isom(S,g_1) = \phi Isom(S,g_2) \phi^{-1} = Isom(S, \phi^* g_2).$$
This shows that unimodular, completely solvable Lie groups are indeed homogeneous maximal symmetry spaces.

\section{Proof of main result for general solvable, unimodular group}

To prove this result, we start by adjusting our metric so as to enlarge the isometry group to one which is the isometry group of a completely solvable, unimodular group.  Then we show that the completely solvable group obtained is unique up to conjugation and use this to prove that there is one largest isometry group for $R$ up to conjugation.

\subsection{Enlarging the isometry group to find some completely solvable group}

\begin{prop}\label{prop: step 1}  Let $R$ be a simply-connected solvmanifold with left-invariant metric $g$.  There exists another left-invariant metric  $g'$ such
	\begin{enumerate}
	\item $Isom(R,g) < Isom(R,g')$, and
	\item $Isom(R,g')$ contains a transitive, completely solvable group $S$.
	\end{enumerate}
\end{prop}

\begin{proof}[Proof of \ref{prop: step 1}]     From the work of Gordon-Wilson (see Lemma \ref{lemma: isom group of unimodular solv} above), we have the existence of a transitive, solvable subgroup $R'' < Isom(R,g)$ such that 
	$$Isom(R,g) = AlgIsom(R'',g) = C \ltimes R''$$
where $C$ consists of orthogonal automorphisms of $\mathfrak r''$.  Here $R''$ is the group in standard position in $Isom(R,g)$.

There exists a maximal compact subalgebra $K$ of $Aut(R'')$ containing $C$.  Choose any inner product $g'$ on  $\mathfrak r''$ so that $K$ consists of orthogonal automorphisms.  Then we immediately see that  that $Isom(R'', g')  > K\ltimes R''$.   Applying Lemma \ref{lemma: changing transitive group of isoms}, we see that $R$ acts transitively by isometries on $(R'',g')$ and so this new left-invariant metric $g'$ on $R''$ gives rise to a left-invariant metric on $R$.  This choice of $g'$ satisfies part $(i)$.

To finish, we show that $Isom(R,g')$ contains a completely solvable group $S$ which acts transitively.  Consider the group $Ad(R'')$ as a subgroup of $Aut(R'')$.  This group is a normal, solvable subgroup and so is a subgroup of the radical $Rad(Aut(R''))$ of $Aut(R'')$.

As $Rad(Aut(R''))$ is an  algebraic group, it has an algebraic Levi decomposition 
	$$Rad(Aut(R'')) = M \ltimes N$$
where $M$ is a maximal reductive subgroup and $N$ is the unipotent radical (see \cite{Mostow:FullyReducibleSubgrpsOfAlgGrps}).  Furthermore, the group $M$ is abelian and decomposes as $M = M_K M_P$, where $M_K$ is a compact torus and $M_P$ is a split torus.  As maximal compact subgroups are all conjugate and $Rad(Aut(R''))$ does not change under conjugation, we may assume that $M_K <K$.  So given $X\in \mathfrak r''$, we may write
	$$ad~X = K_X + P_X + N_X$$
where $K_X\in Lie~M_K$, $P_X\in Lie~M_P$, and $N_X\in \mathfrak n = Lie~N$.

Now define the set $\mathfrak s \subset \mathfrak r'' \rtimes Lie~M_K \subset Lie\ Isom(R'',g')$ as
	$$\{X-K_X \ | \ X\in \mathfrak r'' \}$$
Since the $K_X$ all commute, the nilradical of $\mathfrak r''$ is contained in $\mathfrak s$,  and derivations of $\mathfrak r''$ are valued in the nilradical, we see that $\mathfrak s$ is a solvable Lie algebra.   

Observe that $\mathfrak s$ is completely solvable (this follows as in the proof of Proposition \ref{prop: max comp solv ideal}) and $S$ acts transitively (via Lemma \ref{lemma: changing transitive group of isoms}).

\end{proof}

As completely solvable groups are always in standard position (Proposition \ref{prop:  compl solv are in std position}), we see that $R$ is a modification of the group $S$ and that 
	$$Isom(R,g) < Isom(R,g') = Isom(S,g') = S \rtimes C$$
where $C$ is the compact group of orthogonal automorphisms of $\mathfrak s$, relative to $g'$.  Let $K$ denote a maximal compact group of automophisms of $\mathfrak s$ and $g''$ an inner product on $\mathfrak s$ such that $Isom(S,g'') = S\rtimes K$.  Applying Lemma \ref{lemma: changing transitive group of isoms}, we see that $R$ acts transitively by isometries on $(S,g'')$ and that
	$$Isom(R,g) < Isom(R,g'')$$
In this way, we have found an isometry group  $Isom(R,g'')$ which is a maximal isometry group for $S$ and so by Theorem \ref{thm: max symm for compl solv unimod} cannot be any larger.

This is a reasonable candidate for maximal isometry group for $R$; we verify this in the sequel.

\subsection{The uniqueness of $S$}

The group $S$, constructed  above, depends on several choices made based on various initial and chosen metrics.   More precisely, one starts with metric $g$, does two modifications to obtain the group $R''$, then changes the metric to some $g'$ to extract the group $S$.

If one were to start with a different metric $h$ on $R$, then $R''$ would certainly be different and so it is unclear, a priori,  how would the resulting $S$ for $h$ compares  to the first one.  Surprisingly, they must be conjugate via $Aut(R)$.

\begin{prop}\label{prop:  s contained in r rtimes k}  There exists a maximal compact subalgebra $\mathfrak k$ of $Der(\mathfrak r)$  such that $\mathfrak s$ is the maximal completely solvable ideal of $\mathfrak r \rtimes \mathfrak k$.
\end{prop}

Before proving this proposition, we use it to show that any two groups $S$ constructed from $R$ are conjugate via $Aut(R)$.

Let $g$ and $h$ be two different metrics on $R$ with associated completely solvable algebras $\mathfrak s_g$ and $\mathfrak s_h$, respectively.  Let $\mathfrak k_g$ and $\mathfrak k_h$ be the compact algebras as in the proposition above for $g$ and $h$, respectively.  As the maximal compact of a group is unique up to conjugation, we have some $\phi \in Aut(R)$ such that $\mathfrak k_g = \phi \mathfrak k_h \phi^{-1}$.  This implies
	$$\phi \mathfrak s_h \phi^{-1} \subset \mathfrak r \rtimes   \phi \mathfrak k_h \phi^{-1} = \mathfrak r \rtimes  \mathfrak k_g.$$
As $\phi \mathfrak s_h \phi^{-1}$ is completely solvable and of the same dimension as the maximal completely solvable $\mathfrak s_g$, they must be equal, cf.\ Proposition \ref{prop: max comp solv ideal}.

We now prove Proposition \ref{prop:  s contained in r rtimes k}.

\begin{lemma}\label{lemma: modification of completely solvable is normal modification} Let $\mathfrak s$ be a completely solvable Lie algebra with inner product.  Any modification of $\mathfrak s$ (in its isometry algebra) is a normal modification.
\end{lemma}

\begin{proof}
Let $\mathfrak r = (id + \phi) \mathfrak s$ be a modification of $\mathfrak s$ with modification map $\phi :  \mathfrak s \to N_l(\mathfrak s)$.  To show this is a normal modification, it suffices to show $[\mathfrak s,\mathfrak s]\subset Ker~\phi$ by Proposition 2.4 of \cite{GordonWilson:IsomGrpsOfRiemSolv}.

Denote the nilradical of $\mathfrak{s}$ by $\mathfrak{n(s)}$.  As every derivation of $\mathfrak s$ takes its value in $\mathfrak{n(s)}$, we can decompose $\mathfrak s = \mathfrak a + \mathfrak{n(s)}$ where $\mathfrak a$ is annihilated by $N_l(\mathfrak s)$.  As $\phi$ is linear, to show $[\mathfrak s,\mathfrak s] \subset Ker~\phi$, it suffices to show $[\mathfrak a, \mathfrak a], [\mathfrak a,\mathfrak{n(s)} ], [\mathfrak{n(s)},\mathfrak{n(s)}] \subset Ker~\phi$.

Take $X,Y\in \mathfrak a$.  By the construction of $\mathfrak a$ and Proposition 2.4 (i) of \cite{GordonWilson:IsomGrpsOfRiemSolv}, we have
	$$ [X,Y] = \phi(X)Y - \phi(Y) X + [X,Y] = [\phi(X) +X ,\phi(Y)+Y] \in Ker~\phi$$
Now consider $X\in \mathfrak a$ and $Y\in\mathfrak{n(s)}$.  As above, the following is contained in $Ker~\phi$
	$$[\phi(X)+X,\phi(Y)+Y] = \phi(X)Y - [X,Y]$$
that is, $ad~(\phi(X) + X) : \mathfrak{n(s)} \to Ker~\phi$.

Since every derivation of $\mathfrak s$ takes its image in $\mathfrak{n(s)}$, and $\mathfrak r \subset N_l(\mathfrak s)\ltimes \mathfrak{s}$, we see that $\mathfrak{n(s)}$ is stable under $D=ad~(\phi(X) + X) = \phi(X) + ad~X$.  Denoting the generalized eigenspaces of $D$ on $\mathfrak{n(s)}^\mathbb C$ by $V_\lambda$, we have
	$$\mathfrak{n(s)} = \bigoplus  (V_\lambda \oplus  V_{\bar{\lambda} }) \cap \mathfrak{n(s)}$$
Each summand is invariant under both $\phi(X)$ and $ad~X$ as these commute.  Further, if $\lambda = a + bi$, then on $V_\lambda$ we have $\phi(X)^2 = -b^2 Id$ and $ad~X$ can be realized as an upper triangular matrix whose diagonal is $a~Id$.  Observe that $Ker~D = Ker~\phi(X) \cap Ker~ad~X$, so if $b\not = 0$, then we see that $D$ is non-singular on $V_\lambda$ and $V_{\bar \lambda}$.  This implies $V_\lambda = D(V_\lambda)$ and $V_{\bar \lambda} = D(V_{\bar \lambda})$, which implies
	$$Im(ad~X|_{(V_\lambda \oplus  V_{\bar{\lambda} }) \cap \mathfrak{n(s)} }      )  \subset    (V_\lambda \oplus  V_{\bar{\lambda} }) \cap \mathfrak{n(s)} \subset Im(D|_{\mathfrak{n(s)}}) \subset Ker~\phi$$
If $b=0$, then $V_\lambda = V_{\bar \lambda}$ and $D|_{V_\lambda} = ad~X|_{V_\lambda}$.  Which implies
	$$ad~X|_{V_\lambda} \subset Ker~\phi$$
All together, this proves $[\mathfrak a,\mathfrak{n(s)}] \subset Ker~\phi$.

To finish, one must show $[\mathfrak{n(s)},\mathfrak{n(s)}]\subset Ker~\phi$.  However, as every derivation of $\mathfrak s$ preserves $\mathfrak{n(s)}$, we may restrict our modification to $\mathfrak{n(s)}$ and we have a modification
	$$ \mathfrak n' = (id+\phi) \mathfrak{n(s)} \subset N_l(\mathfrak s) \ltimes \mathfrak{n(s)}$$
Theorem 2.5 of \cite{GordonWilson:IsomGrpsOfRiemSolv} show thats any modification of a nilpotent subalgebra must be a normal modification.  Now Proposition 2.4 (ii d), loc.~cit., implies $[\mathfrak{n(s)},\mathfrak{n(s)}] \subset Ker~\phi$.  This completes the proof of our Lemma.

\end{proof}

\begin{remark} Not all modifications are normal modifications, even in the case of starting with an algebra in standard position.  An example of this can be found in Example 3.9 of \cite{GordonWilson:IsomGrpsOfRiemSolv}.  We warn the reader that there are some typos in that example, the block diagonal matrices of $A$ and $V_1$ should be interchanged.  And then one should replace $A-V_1$ with $A+V_1$ through out the example.
\end{remark}

As explained in the discussion surrounding Lemmas \ref{lemma: norman modification goes both ways} and Lemma \ref{lemma: s is an ideal in r rtimes N_l(r)}, $\mathfrak r$ and $\mathfrak s$ are normal modifications of each other and there is an abelian subalgebra $\mathfrak t$ of the stabilizer subalgebra which normalizes both $\mathfrak r$ and $\mathfrak s$ and for $\mathfrak s \subset \mathfrak r \rtimes \mathfrak t$.  Here $\mathfrak s$ is an ideal.

The proposition follows immediately from the following lemma.

\begin{lemma}  Let $\mathfrak k$ be any maximal compact subalgebra of $Der(\mathfrak r)$ containing $\mathfrak t$.  Then $\mathfrak s$ is a maximal completely solvable  ideal of $\mathfrak r \rtimes \mathfrak k$ (cf. Proposition \ref{prop: max comp solv ideal}).
\end{lemma}

\begin{proof}
By the construction of $\mathfrak s$, it is clearly a complement of $\mathfrak k$ in $\mathfrak r \rtimes \mathfrak k$.  Further, every element of $ad~\mathfrak k$ has purely imaginary eigenvalues on $\mathfrak r\rtimes \mathfrak k$, and so $\mathfrak s$ will be a maximal completely solvable ideal as soon as we show that it is ideal.

As $[\mathfrak s,\mathfrak r] \subset \mathfrak s$ by Lemma \ref{lemma: s is an ideal in r rtimes N_l(r)},  it suffices to show that $\mathfrak s$ is stable under $\mathfrak k$.  However, as every derivation of $\mathfrak r$ takes its image in the nilradical, it suffices to show that $\mathfrak s$ contains the nilradical of $\mathfrak r$.

As stated above,  $\mathfrak r$ being a normal modification of $\mathfrak s$ gives  $\mathfrak r \subset \mathfrak s \rtimes \mathfrak t$  where $\mathfrak t$ consists of skew-symmetric derivations of $\mathfrak s$ and so every element of  $\mathfrak r$ may be written as $X + K$ where $X\in \mathfrak s$ and $K\in \mathfrak t \subset Der(s) \cap \mathfrak{so}(\mathfrak s)$.  One can quickly see, as in the proof of Proposition \ref{prop: max comp solv ideal}, that the eigenvalues of $ad~(X+K)$ are sums of eigenvalues of $ad~X$ and $K$.  Since $ad~X$ has real eigenvalues and $K$ has purely imaginary eigenvalues, we see that $ad~(X+K)$ having only the zero eigenvalue implies $K=0$.  That is, the nilradical of $\mathfrak r$ is contained in $\mathfrak s$.

\end{proof}

Before moving on with the rest of the proof of our main result, we record a consequence of the work done above.

\begin{thm}\label{thm: unique completely solv associated to any solv}  Let $R$ be a solvable Lie group.  Up to isomorphism, there is a single completely solvable group $S$ which can be realized as a modification of $R$.
\end{thm}

\subsection{Maximal symmetry for $R$}  We are now in a position to complete the proof that for a simply-connected, unimodular solvable Lie group there is a single largest isometry group up to conjugation by diffeomorphisms.  In fact, we will see that the diffeomorphism can be chosen to be a composition of an automorphism of $R$ together with an automorphism of $S$.

Starting with a metric $g$ on $R$, we first constructed another metric $g''$ such that 
	$$Isom(R,g)< Isom(R,g'') = S_g \rtimes K,$$
where $S=S_g$ is a completely solvable group (depending on $g$) and $K$ is some maximal compact subgroup of $Aut(S)$.  Let $h$ be another metric on $R$ with corresponding group $S_h$.  From the above, we may replace $h$ with $\phi^*h$ for some $\phi \in Aut(R)$ to assume that $S_h=S_g=S$.

Now, as $K$ is unique up to conjugation in $Aut(S)$, we have the desired result.

\subsection{Proof of corollary \ref{cor: main cor on ricci solitons} }
As in the above, the strategy is to reduce to the setting of completely solvable groups.  We briefly sketch the argument for doing this.

By Theorem 8.2 of \cite{Jablo:HomogeneousRicciSolitons}, any solvable Lie group $R$ admitting a Ricci soliton metric must be a modification of a completely solvable group $S$ which admits a Ricci soliton.  (In fact, those Ricci soliton metrics are isometric.)  From our work above, the modification is a normal modification and so the group $S$ is the same as the group we construct above.  Now the problem is reduced to proving that Ricci soliton metrics on completely solvable Lie groups are maximally symmetric, but this has been resolved - see Theorem 4.1 of \cite{Jablo:ConceringExistenceOfEinstein}.

\section{Constructing $S$ from algebraic data of $R$}
In the above work, we started with a solvable Lie group $R$ and built an associated completely solvable Lie group $S$.  The group $S$ was unique, up to conjugation by $Aut(R)$, but it was built by starting with a metric on $R$, doing modifications to $R$, changing the metric, doing more modifications and then extracting information from the modification $R''$.  We now give a straight-forward description of the group $S$.

Let $K$ be some choice of maximal compact subgroup of $Aut(R)$.  The group $S$ is the simply-connected Lie group whose Lie algebra is the `orthogonal complement' of $\mathfrak k$ in $\mathfrak r \rtimes \mathfrak k$ relative to the Killing form of $\mathfrak r \rtimes \mathfrak k$, i.e.
	\begin{equation}\label{eqn: algebraic description of s}
	\mathfrak s = \{ X\in \mathfrak r \rtimes \mathfrak k \ | \ B(X,Y) = 0 \mbox{ for all } Y\in \mathfrak k \}
	\end{equation}
where $B$ is the Killing form of $\mathfrak r \rtimes \mathfrak k$.  

One can see this quickly by showing that the algebra described in Eqn.~\ref{eqn: algebraic description of s} is  also a maximal completely solvable ideal and then Proposition \ref{prop: max comp solv ideal} shows that is must be $\mathfrak s$.  The details of the proof are similar to work done above  and so we leave them to the diligent reader.

\providecommand{\bysame}{\leavevmode\hbox to3em{\hrulefill}\thinspace}
\providecommand{\MR}{\relax\ifhmode\unskip\space\fi MR }
\providecommand{\MRhref}[2]{%
  \href{http://www.ams.org/mathscinet-getitem?mr=#1}{#2}
}
\providecommand{\href}[2]{#2}

\end{document}